\documentclass[12pt,reqno,a4paper]{amsart}
\usepackage{euscript,fullpage,amsmath,amssymb}
\usepackage[T1]{fontenc}

\newtheorem{theorem}{Theorem}

\newtheorem{lemma}{Lemma}

\renewcommand{\epsilon}{\varepsilon}

\def\Id{\text{\rm Id}}
\def\cA{\mathcal{A}}
\def\cB{\mathcal{B}}

\def\N{\mathbb{N}}
\def\Z{\mathbb{Z}}
\def\R{\mathbb{R}}
\begin{document}

\title{H\"{o}lder continuity of Oseledets splittings
for semi-invertible operator cocycles}

\begin{abstract}
For H\"older continuous cocycles over an invertible, Lipschitz base, we establish the  H\"older continuity of Oseledets subspaces  on compact sets of arbitrarily large measure.
This extends a result of Ara\'{u}jo, Bufetov, and Filip~\cite{ABF} by considering possibly noninvertible cocycles, which in addition may take values in the space of compact operators on a Hilbert space.
As a by-product of our work, we also show that a noninvertible cocycle with nonvanishing Lyapunov exponents exhibits
 nonuniformly hyperbolic behaviour (in the sense of Pesin) on a set of full measure.
\end{abstract}

\begin{thanks}
{DD was supported by an Australian Research Council Discovery Project DP150100017 and in part by the Croatian Science Foundation under the project IP-2014-09-2285.  GF is supported by an ARC Future Fellowship.}
\end{thanks}

\author{Davor Dragi\v cevi\'c}
\address{School of Mathematics and Statistics, University of New South Wales, Sydney NSW 2052, Australia}
\email{d.dragicevic@unsw.edu.au}

\author{Gary Froyland}
\address{School of Mathematics and Statistics, University of New South Wales, Sydney NSW 2052, Australia}
\email{g.froyland@unsw.edu.au}

\maketitle

\section{Introduction}
The celebrated Oseledets multiplicative ergodic theorem (MET)~\cite{Oseledets} plays a fundamental role in the modern theory of dynamical systems. At an abstract level, the MET generalises the notion of eigenvalues and eigenvectors for a single matrix $A\in \mathbb{R}^{d\times d}$ to concatenations of matrices $A(f^{n-1}x)\cdots A(f(x))A(x)$, where $A \colon X \to\mathbb{R}^{d\times d}$ is an invertible  matrix-valued function on a probability space $(X,\mathcal{B},\mu)$, and $f:X\circlearrowleft$.
Under some technical assumptions, the MET guarantees the existence of  a finite set of numbers (called Lyapunov exponents) and subspaces of $\R^d$ (called Oseledets subspaces) which either form a
decomposition or a filtration of $\R^d$ (depending on whether $f$ is invertible or not) such that Lyapunov exponents describe the asymptotic growth of vectors which belong to Oseledets subspaces under the action of~$A$.

 Arguably the most important applications of this result are  in the area of smooth dynamics. For example, the proof of~MET  initiated the study of nonuniformly hyperbolic dynamical systems; that is, systems with nonzero Lyapunov exponents with respect to some smooth invariant probability measure.  
Since the landmark works of  Pesin in 1970s, the theory  of nonuniform hyperbolicity  emerged as an independent, rich and active discipline lying at the heart of dynamical  systems theory.  Among the most important consequences of nonuniform hyperbolicity is the  existence of stable invariant manifolds and their absolute continuity  property (see~\cite{Pesin}). The theory also describes the ergodic properties of a dynamical system with a finite invariant measure that is absolutely continuous with respect to the volume, and it expresses the Kolmogorov-Sinai entropy in terms of the Lyapunov exponents by Pesin's entropy formula (see~\cite{Pesin}). Furthermore, combining the nonuniform hyperbolicity with the nontrivial recurrence guaranteed by the existence of a finite invariant measure, the work of Katok~\cite{Katok} revealed a rich and complicated orbit structure, including an exponential growth rate for the number of periodic points measured in terms of the topological entropy, and the approximation of the entropy of  an invariant measure by uniformly hyperbolic horseshoes. More recently, Barreira,  Pesin and Schmeling~\cite{BPS} discovered a striking relation between this theory and  a dimension theory of dynamical systems by resolving the long standing Eckmann-Ruelle conjecture.  We refer to~\cite{BP} for further references and  a detailed exposition  of this theory.

Oseledets' MET has not only been reproved in many different ways, it has also  been generalised several times, including to compact operators on Hilbert spaces by Ruelle \cite{Ruelle}, to compact operators on Banach spaces with some continuity conditions on the base $f$ and the dependence of the operators on $x\in X$ by Ma\~n\'e \cite{Mane}, to quasi-compact operators on possibly non-separable Banach spaces with continuity conditions by Thieullen \cite{Thieullen}, and to quasi-compact operators on separable Banach spaces with weaker continuity conditions by Lian and Lu \cite{LL}.
Prior to the publication of \cite{FLQ}, all previous work considered the MET in one (or both) of two flavours:  either there is no invertibility assumption on the base and the linear actions, and one obtains the existence of an equivarient flag or filtration;  or there is an invertibility assumption on \emph{both} the base \emph{and} the linear actions, and one obtains the much stronger outcome of existence of an equivariant splitting.

Froyland, Lloyd, and Quas \cite{FLQ} extended the classical Oseledets multiplicative ergodic theorem by proving that if the base is invertible, a unique Oseledets splitting exists even when the matrices are not necessarily invertible.
\begin{theorem}[\cite{FLQ}]\label{t1}
Let $f:X\circlearrowleft$ preserve an ergodic Borel probability measure $\mu$ and assume that $A:X\to \mathbb{R}^{d\times d}$ satisfies
\begin{equation}\label{G}
\int_X \log^+ \lVert A(x)\rVert \, d\mu (x) <+\infty.
\end{equation}
Then there exist  numbers
\begin{equation}\label{LE}
-\infty \le \lambda_1 < \lambda_2< \ldots < \lambda_k
\end{equation}
and for $\mu$-a.e. $x\in M$ a measurable  decomposition
\begin{equation}\label{OS}
\R^d=E_1(x) \oplus E_2(x) \oplus \ldots \oplus E_k(x)
\end{equation}
such that
\begin{equation}\label{inv}
A(x) E_i(x) \subset E_i(f(x))  \text{(with equality if $\lambda_i > -\infty$)}
\end{equation}
and
\begin{equation}\label{exponents}
\lim_{n\to \infty} \frac 1 n \log \lVert A(f^{n-1}x)\cdots A(fx)A(x)v\rVert=\lambda_i, \quad \text{for $v\in E_i(x)\setminus \{0\}$,  $i\in \{1, \ldots, k\}$.}
\end{equation}
\end{theorem}
Semi-invertible versions of the Oseledets theorem for quasi-compact operator cocycles were developed in \cite{FLQ2} and \cite{GTQ1}, generalising the results of \cite{Thieullen} and \cite{LL}, respectively.
The numbers in~\eqref{LE} are called \emph{Lyapunov exponents} and we will refer to the splitting in~\eqref{OS} as to the  \emph{Oseledets splitting}.  Furthermore, we will say that $E_i(x)$ is the Oseledets subspace that corresponds to a  Lyapunov exponent $\lambda_i$.

It is well-known that in general, in all the above mentioned generalizations of MET, the Oseledets subspaces depend only measurably on base points.
However, it was recently proved by Ara\'{u}jo, Bufetov and Filip~\cite{ABF} that under the assumptions that $f:X\circlearrowleft$ is a Lipschitz map and that  $A:X\to GL(d,\mathbb{R})$ is H\"older continuous, one is able to establish H\"older continuity of the Oseledets subspaces on  compact  sets  of arbitrarily large measure.
The arguments in~\cite{ABF} build on a previous work of Brin~\cite{B} who proved (in a particular case of derivative cocycles) that for Anosov systems, the stable and unstable distributions depend H\"older continuosly everywhere and that the same happens for nonuniformly hyperbolic systems but on a compact set of arbitrarily large measure.

The main objective of this paper is to extend the results from~\cite{ABF} by considering possibly noninvertible cocycles, as well as compact operator cocycles with values in the space of all bounded linear operators acting on some Hilbert space.
In order to describe our main result in  a finite-dimensional case, assume that  $A\colon X \to \R^{d\times d}$ is  a H\"older continuous cocycle  over an invertible Lipschitz transformation $f:X\circlearrowleft$ satisfying~\eqref{G}. We prove that the Oseledets subspaces  in Theorem~\ref{t1} are H\"older  continuous on compact  sets  of arbitrarily large measure.
We emphasize that the lack of  the invertibility causes substantial complications and that  consequently crucial parts of our argument differ from the approach developed in~\cite{ABF}.
In addition, this new setting requires new proofs of versions of some well-known facts from Pesin theory.
For example, Theorem~\ref{t3} establishes upper and lower bounds for the growth of the cocycles when restricted to the subbundles $E(x)$ and $F(x)$ given by
\[
E(x)=E_1(x)\oplus \ldots \oplus E_i(x) \quad \text{and} \quad F(x)=E_{i+1}\oplus \ldots \oplus  E_k(x),
\]
as well as a lower bound on the angle between $E(x)$ and $F(x)$.
This result plays an important role in our arguments but is also of independent interest since it in particular implies that if all Lyapunov exponents are nonzero then the cocycle exhibits a nonuniformly hyperbolic behaviour on a set of full measure.
To the best of our knowledge this result had not yet been established before for semi-invertible cocycles.

We emphasize that semi-invertible cocycles arise in two very important situations from the point of view of applications.
Firstly, the study of Markov chains in a random environment (MCRNs).
Markov chains form the basis of mathematical models for a huge variety of physical, chemical, and biological phenomena, including problems in statistical mechanics, (bio)chemical engineering, epidemic modelling, complex networks, and genetics.
More typically than not, the underlying transition probabilities in the Markov chain model evolve over time according to some external random or time-dependent environment.
Instead of having a single invariant probability measure for a stationary Markov chain, Markov chains in random environments possess a family of (random) invariant measures (see e.g.\ \cite{Cogburn}), which depend on the environment.
In the language of Oseledets' MET, $X$ is the environment, $f:X\circlearrowleft$ describes the evolution of the random environment, and $A:X\to\mathbb{R}^{d\times d}$ is a stochastic matrix-valued function.
The family of random invariant measures are the top Oseledets spaces, corresponding to the leading Lyapunov exponent $\lambda_k=0$.
The stability of these random invariant measures has been explored in \cite{FGT-sd}, where it is shown that under mild assumptions on perturbations to $f$ or $A$ the random invariant measure is continuous {in probability} with respect to the environment.
Theorem \ref{main} in the present work will show that if $f$ is Lipschitz and $x\mapsto A(x)$ is H\"older continuous, then the random invariant measure depends H\"older continuously on the environment configuration $x\in X$ on compact sets of arbitrarily large measure.
These assumptions on $f$ and $A$ are very reasonable for mathematical models of real-world processes and our result provides the assurance that on the vast bulk of the environment space, the time-asymptotic distribution of trajectories of the MCRN varies continuously with the environment.

A second application, which was the motivation for the work \cite{FLQ}, concerns a program to understand time-dependent dynamical systems through transfer operator cocycles.
One begins with a function $x\mapsto T_x$, where each $T_x:M\circlearrowleft$ is a nonlinear map on a smooth Riemannian manifold $M$.
A map cocycle $T_{f^{n-1}x}\circ\cdots\circ T_{fx}\circ T_x$ represents the time-dependent evolution of a nonlinear dynamical system.
For example, let $M$ be a three-dimensional manifold representing the ocean, $X$ be the internal configuration of the ocean (e.g.\ the distribution of pressure gradients), $f$ describe how the internal configuration changes over one day, and $T_x$ describe the motion of water particles over one day given the current configuration is $x$.
Associated with each $T_x$ is a linear operator (the transfer operator, see e.g.\ \cite{Baladi} for definitions) $\mathcal{L}_x:\mathcal{B}\circlearrowleft$, acting on a suitable Banach space $\mathcal{B}$.
Continuing with our ocean example, if $g(z):M\to \mathbb{R}$ describes the distribution of some inert, neutrally buoyant chemical in the ocean at ``time'' $x\in X$, then $(\mathcal{L}_xg)(z)$ is the distribution of the chemical one day later. That is, the transfer operators $\{\mathcal{L}_x\}_{x\in X}$ transform densities in $\mathcal{B}$ to densities in $\mathcal{B}$ just as the maps $\{T_x\}_{x\in X}$ transform points in $M$ to points in $M$.

In many areas of nonlinear dynamics, including fluid dynamics and models of geophysical flow such as the ocean and atmosphere, one is interested in structures that decay to equilibrium very slowly;  so-called \emph{Lagrangian coherent structures} or \emph{coherent sets}.
In fluid dynamics, these represent parts of the fluid that are slow to mix with the rest of the fluid;  in the ocean and atmosphere, these structures have physical manifestations as gyres and eddies, and vortices, respectively.
It turns out that the second largest Lyapunov exponent (the first nontrivial exponent after $\lambda_k=0$) describes the time-asymptotic decay rate of the family of most slowly decaying signed distributions $\{g_x(z)\}_{x\in X}$.
Furthermore, and crucially for applications, these signed distributions are given by the corresponding 2nd Oseledets spaces;  see \cite{FLQ,FLS} for details.
In numerical experiments, the transfer operators $\mathcal{L}_x$ are represented as large stochastic matrices on computers, and the Oseledets spaces are similarly discretised.
Theorem \ref{main} in the present paper states that if $f$ is Lipschitz and the linear actions are H\"older continuous, then the corresponding Oseledets spaces, which describe the coherent structures, are H\"older continuous functions on subsets of the base space $X$ of arbitrarily large measure.
This establishes the important fact that in applications, dramatic changes in coherent structures as a function of the driving configuration are extremely rare.

\section{Semi-invertible cocycles and nonuniform hyperbolicity}
In order to make our arguments more transparent and easier to follow, our presentation is for finite-dimension cocycles.
In the final section we highlight the changes necessary to deal with the infinite-dimensional setting.
A measurable map
$\mathcal A \colon X \times \N_0 \to \mathbb{R}^{d\times d}$, where $\N_0=\{0, 1, 2, \ldots \}$ is said to be a \emph{cocycle} over $f$ if:
\begin{enumerate}
\item $\mathcal A(x,0)=\Id$ for every $x\in X$;
\item $\mathcal A(x, n+m)=\mathcal A(f^n(x), m) \mathcal A(x,n)$ for every $x\in X$ and $n, m\ge 0$.
\end{enumerate}
A map $A\colon X \to \mathbb{R}^{d\times d}$ defined by $A(x)=\cA(x,1)$, $x\in X$ is called a \emph{generator} of a cocycle $\mathcal A$.
For an $f$-invariant set $\Lambda\subset X$, a family of subspaces $E(x)\subset \mathbb{R}^d, x\in \Lambda$ is called $\mathcal{A}$\emph{-invariant} if $A(x)E(x)\subset E(fx)$ for each $x\in \Lambda$.

We will now establish several auxiliary results related to Theorem~\ref{t1} that will be used throughout the paper.  We start with the following lemma.

\begin{lemma}\label{ang}
Assume that $\Lambda$ is an $f$-invariant set and let $E(x)\subset \mathbb{R}^d$ and $F(x)\subset \mathbb{R}^d$, $x\in \Lambda$  be $\cA$-invariant families of subspaces with the property that there exist $\lambda_1 <\lambda_2$, $\epsilon >0$   and measurable functions $C, \tilde C \colon \Lambda \to (0, \infty)$ such that
\begin{enumerate}
\item \begin{equation}\label{angles0}
\lambda_1+3\epsilon \le \lambda_2-2\epsilon;
\end{equation}
\item $E(x)\cap F(x)=\{0\}$ for $x\in \Lambda$;
\item for $x\in \Lambda$, $v\in E(x)\oplus F(x)$ and $n\ge 0$,
\begin{equation}\label{angles1}
\lVert \cA(x,n)v\rVert \le \tilde C(x)e^{(\lambda_2+\epsilon)n}\lVert v\rVert;
\end{equation}
\item for $x\in \Lambda$, $v\in F(x)$ and $n\ge 0$,
\begin{equation}\label{angles2}
\lVert \cA(x,n)v\rVert \ge \frac{1}{C(x)}e^{(\lambda_2-\epsilon)n} \lVert v\rVert;
\end{equation}
\item for $x\in \Lambda$, $v\in E(x)$ and $n\ge 0$,
\begin{equation}\label{angles3}
\lVert \cA(x,n)v\rVert \le C(x)e^{(\lambda_1+\epsilon)n}\lVert v\rVert;
\end{equation}
\item for $x\in \Lambda$ and $m\in \Z$,
\begin{equation}\label{angles4}
\tilde C(f^m(x)) \le \tilde C(x)e^{\epsilon \lvert m\rvert} \quad \text{and} \quad C(f^m(x)) \le C(x)e^{\epsilon \lvert m\rvert}.
\end{equation}
\end{enumerate}
Then, there exists  a measurable function $K\colon \Lambda \to (0, \infty)$ satisfying
\begin{equation}\label{KM} K(f^m(x)) \le K(x)e^{5\epsilon \lvert m\rvert}, \quad \text{for $x\in \Lambda$ and $m\in \Z$} \end{equation}
and such that
\begin{equation}\label{angles5}
\lVert v_1\rVert \le K(x)\lVert v_1+v_2\rVert \quad \text{and} \quad \lVert v_2\rVert \le K(x)\lVert v_1+v_2\rVert,
\end{equation}
for $v_1\in E(x)$ and $v_2\in F(x)$.
\end{lemma}

\begin{proof}
Let $P(x) \colon E(x) \oplus F(x) \to E(x)$ and $Q(x) \colon E(x) \oplus F(x) \to F(x)$ be projections. Set
\[
\gamma (x)=\inf \{\lVert v_1+v_2\rVert: \ v_1\in E(x), \ v_2\in F(x), \ \lVert v_1\rVert=\lVert v_2\rVert=1 \}.
\]
For any  $v\in E(x) \oplus F(x)$ such that $P(x)v\neq 0$ and $Q(x)v\neq 0$, we have that
\[
\begin{split}
\gamma (x) &\le \bigg{\lVert} \frac{P(x)v}{\lVert P(x)v\rVert}+\frac{Q(x)v}{\lVert Q(x)v\rVert} \bigg{\rVert} \\
&=\frac{1}{\lVert P(x)v\rVert}\bigg{\lVert} P(x)v+\frac{\lVert P(x)v\rVert}{\lVert Q(x)v\rVert}Q(x)v \bigg{\rVert} \\
&= \frac{1}{\lVert P(x)v\rVert}\bigg{\lVert} v+\frac{\lVert P(x)v\rVert-\lVert Q(x)v\rVert}{\lVert Q(x)v\rVert}Q(x)v \bigg{\rVert} \\
&\le \frac{2\lVert v\rVert}{\lVert P(x)v\rVert}.
\end{split}
\]
Hence,
\[
\lVert P(x)v\rVert \le \frac{2}{\gamma(x)}\lVert v\rVert.
\]
We note that  the above inequality is trivially satisfied when $P(x)v=0$. Finally,  if $Q(x)v=0$ then $P(x)v=v$ and we conclude that
\begin{equation}\label{proj}
\lVert P(x)\rVert \le \max \{ 1, 2/ \gamma(x) \}, \quad \text{for $x\in \Lambda$.}
\end{equation}
Similarly,
\begin{equation}\label{proj2}
\lVert Q(x)\rVert \le \max \{ 1, 2/ \gamma(x) \}, \quad \text{for $x\in \Lambda$.}
\end{equation}

Take now arbitrary $v_1\in E(x)$ and $v_2\in F(x)$ such that $\lVert v_1\rVert=\lVert v_2\rVert=1$. By~\eqref{angles1}, \eqref{angles2} and~\eqref{angles3} we have that
\begin{eqnarray}
\nonumber\lVert v_1+v_2\rVert &\ge& \frac{1}{\tilde C(x)e^{(\lambda_2+\epsilon)n}} \lVert \cA(x, n)(v_1+v_2)\rVert  \\
\label{14.5}
&\ge& \frac{1}{\tilde C(x)e^{(\lambda_2+\epsilon)n}} \bigg{(} \frac{1}{C(x)}e^{(\lambda_2-\epsilon)n}-C(x)e^{(\lambda_1+\epsilon)n}\bigg{)},
\end{eqnarray}
for every $n\ge 0$.  Let $n(x)$ be the smallest integer such that
\begin{equation}\label{NX}
\frac{1}{C(x)}e^{(\lambda_2-\epsilon)n(x)}-C(x)e^{(\lambda_1+\epsilon)n(x)} \ge \frac{1}{C(x)}e^{(\lambda_2-2\epsilon)n(x)}
\end{equation}
or equivalently
\begin{equation}
\label{15.5}e^{(\lambda_2-\epsilon)n(x)}-C(x)^2e^{(\lambda_1+\epsilon)n(x)} \ge e^{(\lambda_2-2\epsilon)n(x)}.
\end{equation}
By (\ref{14.5}) and \eqref{NX},
\[
\gamma(x) \ge \frac{1}{\tilde C(x)e^{(\lambda_2+\epsilon)n(x)}} \cdot \frac{1}{C(x)}e^{(\lambda_2-2\epsilon)n(x)}
\]
and thus
\begin{equation}\label{angles6}
\frac{2}{\gamma(x)} \le 2C(x)\tilde C(x)e^{3\epsilon n(x)}.
\end{equation}
Finally, we claim that  that $n(f^m(x)) \le n(x)+\lvert m\rvert$ for each $x\in \Lambda$ and  $m\in \Z$.  Indeed, using~\eqref{angles0} and~\eqref{angles4},  we have that
\begin{eqnarray*}
\lefteqn{e^{(\lambda_2-\epsilon)(n(x)+\lvert m\rvert)}-C(f^m(x))^2e^{(\lambda_1+\epsilon)(n(x)+\lvert m\rvert)}} \\ &\ge& e^{(\lambda_2-\epsilon)n(x)} \cdot e^{(\lambda_2-\epsilon)\lvert m\rvert}-C(x)^2e^{(\lambda_1+\epsilon)n(x)} \cdot e^{(\lambda_1+3\epsilon)\lvert m\rvert} \\
&\ge& e^{(\lambda_2-\epsilon)n(x)} \cdot e^{(\lambda_2-\epsilon)\lvert m\rvert} -C(x)^2e^{(\lambda_1+\epsilon)n(x)} \cdot e^{(\lambda_2-2\epsilon)\lvert m\rvert} \\
&\ge & e^{(\lambda_2-2\epsilon)n(x)} \cdot e^{(\lambda_2-2\epsilon)\lvert m\rvert} \qquad\mbox{by (\ref{15.5})}\\
&\ge &e^{(\lambda_2-2\epsilon)(n(x)+\lvert m\rvert)}.
\end{eqnarray*}
In order to complete the proof of the lemma, we are going to show that the function $K(x)=\max \{ 1, 2C(x)\tilde C(x)e^{3\epsilon n(x)} \}$ satisfies~\eqref{KM} and~\eqref{angles5}. We note that~\eqref{angles5} follows directly from~\eqref{proj}, \eqref{proj2} and~\eqref{angles6}.  Moreover, using~\eqref{angles4}, we have that
\[
 C(f^m(x))\tilde C(f^m(x))e^{3\epsilon n(f^m(x))} \le C(x) \tilde C(x)e^{3\epsilon n(x)} \cdot e^{5\epsilon \lvert m\rvert},
\]
for each $x\in \Lambda$ and $m\in \Z$, which readily implies that~\eqref{KM} holds. 
\end{proof}

Suppose that the Lyapunov exponents of the cocycle $\cA$ are given by~\eqref{LE}. Then, for each $i\in \{1, \ldots, k\}$, we can associate to~\eqref{OS} a new decomposition of $\R^d$ as
\begin{equation}\label{OS2}
\R^d=\bigg{(}\bigoplus_{j\le i} E_j(x)\bigg{)} \oplus \bigg{(}\bigoplus_{j> i} E_j(x)\bigg{)}.
\end{equation}
The following result establishes exponential bounds for $\cA$ along the two subspaces forming the decomposition~\eqref{OS2} as well as for angles between them. For invertible cocycles such a result is well-known (see Theorem 3.3.1 in~\cite{BP} for example).
A major difficulty in adapting the arguments in~\cite{BP} is that they rely heavily on the well-known fact that the angles between Oseledets subspaces in the standard (invertible)  MET exhibit a subexponential growth along each trajectory.  On the other hand, to the best of our knowledge no such statement was established in relation to the semi-invertible version of MET stated in Theorem~\ref{t1}. This forces us to develop an argument (based on Lemma~\ref{ang}), which is  completely different from the one in~\cite{BP}, to first establish exponential bounds for $\cA$ along the subspaces in~\eqref{OS2} and then use this to deduce an appropriate bound for the angle between those subspaces.
\begin{theorem}\label{t3}
Let $\cA$ be a cocycle over $f$ satisfying~\eqref{G} with Lyapunov exponents as in~\eqref{LE} and take $i\in \{1, \ldots, k\}$.
Let
\[
E^1(x)=\bigoplus_{j=1}^iE_j(x)\quad \text{and} \quad E^2(x)=\bigoplus_{j=i+1}^kE_j(x).
\]
Then, there exists a Borel set $\Lambda \subset X$ such that $\mu(\Lambda)=1$ and for each $\epsilon >0$ there are   measurable functions $C, K \colon \Lambda \to (0, \infty)$ with the property that for every $x\in \Lambda$ we have that:
\begin{enumerate}
\item for each $v\in E^1(x)$ and $n\ge 0$,
\begin{equation}\label{1}
\lVert \cA(x,n)v\rVert \le C(x)e^{(\lambda_i+\epsilon)n} \lVert v\rVert,
\end{equation}
where if $i=1$ and $\lambda_1=-\infty$, $\lambda_1$  is replaced by any number that belongs to the  interval $(-\infty, \lambda_2)$;
\item for each $v\in E^2(x)$ and $n\ge 0$,
\begin{equation}\label{2}
\lVert \cA(x, n)v\rVert \ge \frac{1}{C(x)}e^{(\lambda_{i+1}-\epsilon)n}\lVert v\rVert;
\end{equation}
\item for each  $u\in E^1(x)$ and $v\in E^2(x)$,
 \begin{equation}\label{A} \lVert u\rVert \le K(x)\lVert u+v\rVert  \quad \text{and} \quad \lVert v\rVert \le  K(x)\lVert u+v\rVert, \end{equation}
\item for each $n\in \Z$,
\begin{equation}\label{3}
C(f^n(x)) \le C(x)e^{\epsilon \lvert n\rvert} \quad \text{and} \quad K(f^n(x)) \le K(x)e^{\epsilon \lvert n\rvert}.
\end{equation}
\end{enumerate}
\end{theorem}

\begin{proof}
\quad

\textsl{Step 1 -- Upper bound for growth on $E^1$ and temperedness of the function $C$:}
We begin by establishing property~\eqref{1}.
We start with the following lemma.

\begin{lemma}\label{9216}
We have
\begin{equation}\label{res}
\limsup_{n\to \infty} \frac {1}{n} \log \lVert \cA(x,n)\rvert E^1(x)\rVert \le \lambda_i \quad \text{for $\mu$-a.e. $x\in X$,}
\end{equation}
where $\cA(x,n)\rvert E^1(x)$ denotes the restriction of $\cA(x,n)$ onto $E^1(x)$.
\end{lemma}

\begin{proof}[Proof of the lemma]
Let $\{e_1, \ldots, e_l\}$ be an orthonormal basis for $E^1(x)$.  For each $n\in \N$ let $v_n\in E^1(x)$ be such that
$\lVert v_n\rVert=1$ and $ \lVert \cA(x,n)\rvert E^1(x)\rVert =\lVert \cA(x,n)v_n\rVert$. Furthemore, for  $n\in \N$, write $v_n$ in the form
\[
v_n=\sum_{j=1}^l a_{j,n}e_j,
\]
for $a_{j,n}\in \R$. We note that $\lvert a_{j,n}\rvert =\lvert \langle v_n, e_j\rangle \rvert \le \lVert v_n\rVert \cdot \lVert e_j\rVert=1$ and thus
\begin{equation}\label{TAA}
\lVert \cA(x,n)\rvert E^1(x)\rVert \le \sum_{j=1}^l \lvert a_{j,n}\rvert \cdot \lVert \cA(x,n)e_j\rVert \le \sum_{j=1}^l  \lVert \cA(x,n)e_j\rVert.
\end{equation}
Since $e_j \in E^1(x)$, it follows from~\eqref{exponents} that
\begin{equation}\label{TBB}
\limsup_{n\to \infty} \frac 1n \log \lVert \cA(x,n)e_j\rVert \le \lambda_i, \quad \text{for $j\in \{1, \ldots, l\}$.}
\end{equation}
Finally, we note that~\eqref{TAA} and~\eqref{TBB} readily imply~\eqref{res}.
\end{proof}

It follows from~\eqref{res} that for $\epsilon >0$,
\begin{equation}\label{defD}
D(x):=\sup_{n\ge 0} \{\lVert \cA(x,n)\rvert E^1(x)\rVert  \cdot e^{-(\lambda_i+\epsilon)n}\} <\infty,
\end{equation}
for $\mu$ a.e.\ $x\in X$.
\begin{lemma}
We have
\begin{equation}\label{temp}
\lim_{n \to \pm \infty}\frac{1}{n} \log D(f^n(x))=0 \quad \text{for $\mu$-a.e. $x\in X$.}
\end{equation}
\end{lemma}

\begin{proof}[Proof of the lemma]
For $n\ge 1$, we have
\[
\begin{split}
\lVert \cA(x,n)\rvert E^1(x)\rVert  &\le \lVert \cA(f(x), n-1)\rvert E^1(f(x))\rVert \cdot \lVert A(x)\rvert E^1(x)\rVert  \\
&\le \lVert \cA(f(x), n-1)\rvert E^1(f(x))\rVert \cdot \lVert A(x)\rVert.
\end{split}
\]
By multiplying the above inequality by $e^{-(\lambda_i+\epsilon)n}$, we obtain
\[
e^{-(\lambda_i+\epsilon)n}\lVert \cA(x,n)\rvert E^1(x)\rVert \le e^{-(\lambda_i+\epsilon)(n-1)}\lVert \cA(f(x), n-1)\rvert E^1(f(x))\rVert \cdot e^{-(\lambda_i+\epsilon)}\lVert A(x)\rVert.
\]
Hence,
\[
D(x) \le D(f(x))\cdot \max \{e^{-(\lambda_i+\epsilon)}\lVert A(x)\rVert, 1\}.
\]
It follows from~\eqref{G} that there exists  an nonnegative and  integrable function $\psi \colon X \to \R$ such that
\begin{equation}\label{X}
\log D(x)- \log D(f(x))\le \psi(x).
\end{equation}
Set
\[
\tilde D(x)=\log D(x)- \log D(f(x)).
\]
We note that
\begin{equation}\label{Z1}
\frac{1}{n}\log D(f^n(x))=\frac{1}{n}\log D(x)-\frac{1}{n}\sum_{j=0}^{n-1}\tilde D(f^j(x)),
\end{equation}
for each $x\in X$ and $n\in \N$. By~\eqref{X}, we have that $\tilde D^+$ is integrable. Hence, we can apply the  Birkhoff ergodic theorem (see~p.539 \cite{A}) and conclude that there exists $a\in [-\infty, \infty)$ such that
\begin{equation}\label{Z2}
\lim_{n\to \infty} \frac{1}{n} \sum_{j=0}^{n-1}\tilde D(f^j(x))=a,
\end{equation}
for $\mu$-a.e. $x\in X$. It follows from~\eqref{Z1} and~\eqref{Z2} that
\[
\lim_{n\to \infty} \frac{1}{n}\log D(f^n(x))=-a.
\]
On the other hand, since $\mu$ is $f$-invariant, for any $c>0$ we have that
\[
\lim_{n\to \infty} \mu (\{x\in X: \log D(f^n(x))/n \ge c\})=\lim_{n\to \infty} \mu (\{x\in X: \log D(x) \ge nc\})=0,
\]
which immediately implies that $a\ge 0$. Thus,
\[
\lim_{n\to \infty} \frac{1}{n} \log D(f^n(x))\le 0.
\]
Since $D(x)\ge 1$ for $\mu$ a.e.\ $x\in X$  by (\ref{defD}), we conclude that~\eqref{temp} holds when $n\to \infty$.

Now we establish~\eqref{temp} for the case $n\to -\infty$. Set
\[
 D'(x)=\log D(f^{-1}(x))- \log D(x).
\]
Obviously,
\begin{equation}\label{Z11}
\frac{1}{n}\log D(f^{-n}(x))=\frac{1}{n}\log D(x)+\frac{1}{n}\sum_{j=0}^{n-1} D'(f^{-j}(x)),
\end{equation}
for each $x\in X$ and $n\in \N$.  By~\eqref{X}, we have that $D'^+$ is integrable. Hence, we can apply the Birkhoff ergodic theorem  and conclude that there exists $a\in [-\infty, \infty)$ such that
\begin{equation}\label{Z22}
\lim_{n\to \infty} \frac{1}{n} \sum_{j=0}^{n-1} D'(f^{-j}(x))=a,
\end{equation}
for $\mu$-a.e. $x\in X$. Once can now proceed as in the previous case and obtain that $a=0$, which implies~\eqref{temp}.
\end{proof}
It follows from~\eqref{temp} and Proposition 4.3.3(ii) in~\cite{A} that there exists a nonnegative and measurable function $C$ defined on a set of full-measure  satisfying the first inequality in~\eqref{3} such that $D(x)\le C(x)$ which together with~\eqref{defD} implies that~\eqref{1} holds.

\textsl{Step 2 -- Lower bound for growth on $E^2$ and temperedness of the function $1/C$:}
We now show (\ref{2}).
By Theorem~\ref{t1}, the cocycle $\cA$ is invertible along the subbundle $E^2$.
We would like to apply Step 1 to the inverse of the  cocycle obtained by restricting $\cA$ onto the subbundle $E^2$  to conclude that~\eqref{2} holds for some function $C$ satisfying the first inequality in~\eqref{3} on a set of full measure.
In order to do this, we first require the integrability condition of Lemma \ref{inv_integ}.
The arguments in the proof of Lemma \ref{11216} are partly inspired by those in the proof of Corollary 3.8 \cite{blumenthal_young}.
\begin{lemma}\label{11216}
\label{inv_integ}
We have
\[
\int_X \log^+ \lVert (A(x)\rvert E^2(x))^{-1}\rVert \, d\mu(x) < \infty.
\]
\end{lemma}

\begin{proof}[Proof of the lemma]
Take an arbitrary $v\in E^2(x)$ such that $\lVert v\rVert=1$ and find an orthonormal basis $\{v_1, \ldots, v_m\}$ of $E^2(x)$ such that $v_1=v$. Then,
\[
\lvert \det (A(x)\rvert E^2(x))\rvert \le \lVert A(x)v\rVert \cdot \prod_{i=2}^m \lVert A(x)v_i\rVert \le \lVert A(x)v\rVert \cdot \lVert A(x)\rVert^{m-1}.
\]
Hence,
\begin{equation}\label{EA}
\lVert A(x)v\rVert \ge \lVert A(x)\rVert^{1-m} \cdot \lvert \det (A(x)\rvert E^2(x))\rvert .
\end{equation}
Moreover, by~\eqref{EA} we have
\[
\lVert (A(x)\rvert E^2(x))^{-1}\rVert=\sup_{w\in E^2(fx), \|w\|=1}\|A(x)^{-1}w\|=\sup_{v\in E^2(x), \lVert v\rVert=1}\frac{1}{\lVert A(x)v\rVert} \le \frac{\lVert A(x)\rVert^{m-1}}{\lvert \det (A(x)\rvert E^2(x))\rvert},
\]
and therefore
\[
\log \lVert (A(x)\rvert E^2(x))^{-1}\rVert \le (m-1)\log \lVert A(x)\rVert -\log \lvert \det (A(x)\rvert E^2(x))\rvert.
\]
In  view of~\eqref{G}, setting
\begin{equation}\label{psi}
\psi(x)=\log \lvert \det (A(x)\rvert E^2(x))\rvert,
\end{equation}
it remains to prove $\psi^-\in L^1(\mu)$.
We first note that $\psi (x)\le \log \lVert A(x)\rVert^m=m \log \lVert A(x)\rVert$, which together with~\eqref{G} implies that $\psi^+\in L^1(\mu)$. It follows from Birkhoff's ergodic theorem that there exists $a\in \R \cup \{-\infty\}$ such that
\[
a=\lim_{n\to \infty} \frac 1 n\sum_{i=0}^{n-1} \psi (f^i(x))=\lim_{n\to \infty} \frac 1n \sum_{i=1}^n \psi (f^{-i}(x)),
\]
for $\mu$-a.e. $x\in X$. We note that
\[
\begin{split}
a=\lim_{n\to \infty} \frac 1n \sum_{i=1}^n \psi (f^{-i}(x)) &=\lim_{n\to \infty} \frac 1n \sum_{i=1}^n \log \lvert \det (A(f^{-i}(x))\rvert E^2(f^{-i}(x)))\rvert \\
&=\lim_{n\to \infty} \frac 1n \log \lvert \det (\cA(f^{-n}(x), n\rvert E^2(f^{-n}(x))) \rvert \\
&= \lim_{n\to \infty} \frac 1n \log \frac{1}{\lvert \det (\cA(f^{-n}(x), n\rvert E^2(f^{-n}(x)))^{-1} \rvert}.
\end{split}
\]
Let again $\{v_1, \ldots, v_m\}$ be an orthonormal basis for $E^2(x)$. Then by going backwards in~\eqref{LE} (which is possible by Lemma 20 \cite{FLQ2}),
we can find $\lambda \in \R$ such that for sufficiently large $n$,
\[
\lvert \det (\cA(f^{-n}(x), n)\rvert E^2(f^{-n}(x)))^{-1} \rvert \le \prod_{i=1}^m \lVert (\cA(f^{-n}(x), n)\rvert E^2(f^{-n}(x)))^{-1}  v_i\rVert  \le e^{mn\lambda},
\]
which implies that $a\ge -m\lambda$ and thus $a\in \R$. Moreover, by Kingman's subadditive ergodic theorem
\[
a=\inf_{n\in  \N} \frac 1n \int_X \log \lvert \det (\cA(f^{n-1}(x), n)\rvert E^2(x))\rvert   \, d\mu (x) \le \int_X \log \lvert \det (A(x)\rvert E^2(x))\rvert \, d\mu(x),
\]
which implies the integrability of $\psi^-$.
\end{proof}

To finish this step, we now establish the existence of function $C$ satisfying~\eqref{2} and the first inequality in~\eqref{3}. By Theorem~\ref{t1}, the map  $A(x)$ is invertible along the direction $E^2(x)$ and we  will denote the inverse of this map by $A^{-1}(x)$. Let $\mathcal B$ be a cocycle over $f^{-1}$ defined on a subbundle $E^2(x)$ with generator $A^{-1} \circ f^{-1}$. It follows from Theorem~\ref{t1}, Lemma~\ref{11216} and  Lemma 20~\cite{FLQ2} that the  Lyapunov exponents of the cocycle $\cB$ are given by
\[
-\lambda_k < \ldots < -\lambda_{i+1}.
\]
Furthermore,  $E_j(x)$ is the Oseledets subspace corresponding to $-\lambda_j$ for $i+1\le j \le k$.
We can now apply Step 1 to $\cB$ to
conclude that that there exists a function $C\colon \Lambda \to (0, \infty)$ such that
\begin{equation}\label{Y1}
\lVert \mathcal B(x, n)v\rVert \le C(x)e^{(-\lambda_{i+1}+\frac{\epsilon}{2})n}, \quad \text{for $x\in \Lambda$, $n\ge 0$ and $v\in E_{j+1}(x)\oplus \ldots \oplus E_k(x)$}
\end{equation}
and
\begin{equation}\label{Y2}
C(f^m(x))\le C(x)e^{\frac{\epsilon}{2}  \lvert m\rvert}, \quad \text{for $x\in \Lambda$ and $m\in \Z$.}
\end{equation}
It follows readily from~\eqref{Y1} and~\eqref{Y2} that~\eqref{2} holds.


\textsl{Step 3 -- Lower bound for $K$ and temperedness of $K$: }

The existence of function $K$ satisfying~\eqref{A} and~\eqref{3} follows by applying Lemma~\ref{ang} successively.
For $i=k-1$, it is sufficient to apply Lemma~\ref{ang} to $E(x)=E_1(x)\oplus \ldots \oplus E_{k-1}(x)$ and $F(x)=E_k(x)$ using the  properties~\eqref{1} (both for $E(x)$ and $E(x)\oplus F(x)$) and~\eqref{2} from Steps 1 and 2.
For $i=k-2$, we again apply Lemma~\ref{ang} to $E(x)=E_1(x)\oplus \ldots \oplus E_{k-1}(x)$ and $F(x)=E_k(x)$ and obtain a  function $K_1$ as in the statement of Lemma~\ref{ang}.
Further, we apply Lemma~\ref{ang} to $E(x)=E_1(x)\oplus \ldots \oplus E_{k-2}(x)$ and $F(x)=E_{k-1}(x)$ and obtain a function $K_2$ as in the statement of Lemma~\ref{ang}. Take now an arbitrary $v\in E_1(x)\oplus \ldots \oplus E_{k-2}(x)$ and $w\in E_{k-1}(x)\oplus E_k(x)$, $w=w_1+w_2$, $w_1\in E_{k-1}(x)$, $w_2\in E_k(x)$. By~\eqref{angles5},
\begin{equation}\label{I1}
\lVert v\rVert \le K_2(x)\lVert v+w_1\rVert \le K_1(x)K_2(x)\lVert v+w\rVert.
\end{equation}
Similarly,
\[
\lVert w_2\rVert \le K_1(x)\lVert v+w\rVert
\]
and
\[
\lVert w_1\rVert \le K_2(x)\lVert v+w_1\rVert \le K_1(x)K_2(x)\lVert v+w \rVert.
\]
Hence,
\begin{equation}\label{I2}
\lVert w\rVert \le 2\max \{K_1(x), K_1(x)K_2(x)\} \lVert v+w\rVert.
\end{equation}
It follows from~\eqref{I1} and~\eqref{I2} that~\eqref{A} holds for $K(x)=2\max \{K_1(x), K_1(x)K_2(x)\}$, which in view of~\eqref{KM}  satisfies the second inequality in~\eqref{3}  with $\epsilon$ replaced by some $a\epsilon$ for some $a>0$ (this is possible since $\epsilon >0$ can be made arbitrarily small). Proceeding inductively, we can establish the appropriate bounds for the angle in the general case when $1\le i\le k$.
\end{proof}

As we have already noted in the introduction, Theorem~\ref{t3} plays a central role in the proof of our main result, however, it is also a result of independent interest. For example, it shows that the noninvertible cocycles with all nonzero Lyapunov exponents are nonuniformly hyperbolic in the sense of Pesin (see~\cite{BP} for details) on a set of full-measure.
Furthermore, it shows that the notion of a nonuniform exponential dichotomy for not necessarily invertible discrete time dynamics (introduced by Barreira and Valls in~\cite{BV}) is ubiquitous in the context of ergodic theory.
In order to formulate an explicit result, we recall that we say that a sequence  $(A_n)_{n\in \Z}$ of operators on $\R^d$ admits a \emph{nonuniform exponential dichotomy} if:
\begin{enumerate}
\item
there exist projections $P_n\colon \R^d \to \R^d$ for each $n\in \Z$ satisfying
\[
A_n P_n=P_{n+1}A_n
\]
for $n\in \Z$ such that each map
\[
A_n \rvert \ker P_n \colon \ker P_n \to \ker P_{n+1}
\]
is invertible;
\item
there exist a constants $D, \lambda >0$ and $\epsilon \ge 0$
such that
\[
\lVert \cA(m,n)P_n \rVert \le De^{-\lambda (m-n)+ \epsilon \lvert n\rvert} \quad \text{for} \quad m\ge n
\]
and
\[
\lVert \cA(m,n)Q_n \rVert \le De^{-\lambda (n-m)+ \epsilon \lvert n\rvert }\quad \text{for} \quad m\le n,
\]
where $Q_n=\Id-P_n$ and
\[
\cA(m,n)=\left(\cA(n,m)\rvert \ker P_m\right)^{-1}\colon \ker P_n \to \ker P_m
\]
for $m<n$.
\end{enumerate}

The following result is a direct consequence of Theorem~\ref{t3}.

\begin{theorem}
\label{mainb}
Let $\cA$ be a cocycle satisfying~\eqref{G} with nonvanishing Lyapunov exponents.
Then, there exists a Borel set $\Lambda \subset X$ of full $\mu$-measure such that for each $x\in \Lambda$  the sequence $(A_n)_{n\in \Z}$ defined by $A_n=A(f^n(x))$, $n\in \Z$ admits a nonuniform exponential dichotomy.
\end{theorem}

\section{An  Oseledets splitting of the adjoint cocycle}
The other crucial ingredient in the proof of H\"older continuity of the Oseledets splitting (Theorem \ref{main}) is the use of the adjoint cocycle.
Suppose that $\cA$ is a cocycle over $f$ whose generator $A$ satisfies~\eqref{G}. Moreover, assume that  the Lyapunov exponents of $\cA$ and the corresponding Oseledets decomposition are given by~\eqref{LE} and~\eqref{OS}.  We denote by  $\mathcal A^*$  the cocycle over $f^{-1}$ with generator $A^*\circ f^{-1}$. The following result  identifies Lyapunov exponents and the Oseledets splitting of the cocycle $\mathcal A^*$. We note that this  theorem is well-known for invertible matrix cocycles (see Theorem 5.1.1 \cite{A} for example).

\begin{theorem}\label{t2}
The Lyapunov exponents of the cocycle $\cA^*$ are given by~\eqref{LE}.  Furthermore, the Oseledets subspace that corresponds to  $\lambda_i$ is given by
\begin{equation}\label{Fi}
\bigg{(} \bigoplus_{j\neq i} E_j(x) \bigg{)}^{\perp}.
\end{equation}
\end{theorem}
\begin{proof}
Let $F_i(x)$ be a subspace of $\R^d$ given by~\eqref{Fi}. It follows from~\eqref{inv} that
\[
A(x)\bigg{(}\bigoplus_{j\neq i} E_j(x) \bigg{)} \subset \bigoplus_{j\neq i} E_j(f(x)),
\]
which readily implies that $A^*(f^{-1}(x))F_i(x) \subset F_i(f^{-1}(x))$ for each $i\in \{1, \ldots, k\}$.
 Furthermore, the subspaces $F_i(x)$ form a direct sum. Indeed, assume that
\[v\in F_i(x) \cap (F_1(x)+ \ldots F_{i-1}(x)+F_{i+1}(x) + \ldots + F_k(x)), \] and   write $v$ in the form $v=v_1+\ldots +v_k$, where $v_j \in E_j(x)$ for $j=1, \ldots, k$. Since \[ v\in F_i(x) \quad  \text{and} \quad  v_1+\ldots +v_{i-1}+v_{i+1}+\ldots +v_k\in \bigoplus_{j\neq i}E_j(x),\] we conclude that $\langle v, v_1+\ldots +v_{i-1}+v_{i+1}+\ldots +v_k \rangle=0$.
On the other hand, since \[ v\in F_1(x)+ \ldots F_{i-1}(x)+F_{i+1}(x) + \ldots + F_k(x)\subset E_i(x)^{\perp},\] we also have that $\langle v, v_i\rangle =0$.  Hence, $\langle v, v\rangle =0$ and $v=0$. We now want to prove that
\[
\R^d=\bigoplus_{i=1}^k F_i(x).
\]
In order to prove the above equality we are going to show that $(\oplus_{i=1}^k F_i(x))^{\perp}=\{0\}$.  Take $v\in (\oplus_{i=1}^k F_i(x))^{\perp}=\cap_{i=1}^kF_i(x)^\perp$
and write it in a form $v=v_1+\ldots +v_k$, where $v_i \in E_i(x)$, $i=1, \ldots, k$. Since $v$ and $v_2+\ldots +v_k$ both belong to
$F_1(x)^\perp$, we conclude that $v_1\in F_1(x)^\perp$. However, since the subspaces $E_j(x)$ form a direct sum of $\R^d$ this implies that $v_1=0$. Similarly, we obtain that $v_j=0$ for $j=2, \ldots, k$ and thus $v=0$. In order to complete the proof of the theorem, we are going to show that for $\mu$-a.e. $x\in X$,
\begin{equation}\label{H1}
\lim_{n\to \infty} \frac 1n \log \lVert \cA^*(x,n)u\rVert =\lambda_i, \quad \text{for $u\in F_i(x)\setminus \{0\}$ and $i\in \{1, \ldots, k\}$.}
\end{equation}
Take $i\ge 2$ and $u\in F_i(x)\setminus \{0\}$. We first note that $\cA^*(x,n)=\cA(f^{-n}(x), n)^*$. Hence,
\begin{equation}\label{H2}
 \lVert \cA^*(x,n)u\rVert=\max_{\lVert v\rVert=1} \lvert \langle \cA^*(x,n)u, v\rangle \rvert =\max_{\lVert v\rVert=1} \lvert \langle u, \cA(f^{-n}(x), n)v\rangle \rvert.
\end{equation}
For $v\in E_i(f^{-n}(x))$, $\lVert v\rVert=1$, one has $\cA(f^{-n}(x), n)v \in E_i(x)$.
Since $u\in F_i(x)\setminus\{0\}$, setting
\[
c=c(x)=\frac 1 2\sup \{ \lvert \langle w_1, w_2\rangle \rvert:  w_1\in F_i(x), \ w_2 \in E_i(x), \ \lVert w_1\rVert=\lVert w_2\rVert=1 \},
\] we have
\begin{equation}\label{H3}
\lVert \cA^*(x,n)u\rVert=\max_{\lVert v\rVert=1} \lvert \langle u, \cA(f^{-n}(x), n)v\rangle \rvert \ge c\lVert u\rVert \cdot \lVert \cA(f^{-n}(x), n)v\rVert.
\end{equation}
On the other hand, it follows from Theorem~\ref{t3} that  for each $\epsilon >0$,  there exists a measurable function $C\colon \Lambda \to (0, \infty)$ defined on a Borel set $\Lambda \subset X$ of full $\mu$-measure such that
\begin{equation}\label{H4}
C(f^n(x))\le C(x)e^{\epsilon \lvert n\rvert} \quad \text{for $x\in \Lambda$ and $n\in \Z$,}
\end{equation}
and
\begin{equation}\label{H5}
\lVert \cA(x,n)w\rVert \ge \frac{1}{C(x)}e^{(\lambda_i-\epsilon)n}\lVert w\rVert \quad \text{for $x\in \Lambda$, $w\in E_i(x)$ and $n\ge 0$.}
\end{equation}
It follows from~\eqref{H4} and~\eqref{H5} that
\begin{equation}\label{H6}
\lVert \cA(f^{-n}(x), n)v\rVert \ge \frac{1}{C(x)} e^{(\lambda_i-2\epsilon)n}, \quad \text{for $n\ge 0$.}
\end{equation}
By~\eqref{H2}, \eqref{H3} and~\eqref{H6}, we have
\[
\liminf_{n\to \infty} \frac 1 n\log \lVert \cA^*(x,n)u\rVert \ge \lambda_i-2\epsilon.
\]
Since $\epsilon >0$ was arbitrary, we conclude that
\begin{equation}\label{H7}
\liminf_{n\to \infty}\frac 1n \log \lVert \cA^*(x,n)u\rVert \ge \lambda_i.
\end{equation}
Now take an arbitrary $v\in \R^d$, $\lVert v\rVert=1$ and write it as $v=v_1+\ldots +v_k$, where $v_j\in E_j(f^{-n}(x))$ for $j=1, \ldots, k$. Since $u\in F_i(x)$, we have that $\langle u, \cA(f^{-n}(x), n)v_j\rangle =0$ for $j\neq i$. Hence,
\begin{equation}\label{H8}
\lvert \langle u, \cA(f^{-n}(x), n)v\rangle \rvert =\lvert \langle u, \cA(f^{-n}(x), n)v_i\rangle \rvert \le \lVert u\rVert \cdot \lVert \cA(f^{-n}(x), n)v_i\rVert.
\end{equation}
Furthermore, it follows from Theorem~\ref{t3} that  for each $\epsilon >0$ there exists measurable functions $C, K \colon \Lambda \to (0, \infty)$ defined on a Borel set $\Lambda$ of full $\mu$-measure such that  for every $x\in \Lambda$ and $n\in \Z$ we have:
\begin{equation}\label{H9}
\lVert \cA(x, n)w\rVert \le C(x)e^{(\lambda_i+\epsilon)n}\lVert w\rVert \quad \text{for  $w\in E_i(x)$, $n\ge 0$,}
\end{equation}
\begin{equation}\label{H10}
\lVert w_i\rVert \le K(x)\lVert w_1+w_2\rVert \quad \text{for $i\in \{1, 2\}$,  $w_1\in E_i(x)$ and $w_2\in \oplus_{j\neq i} E_j(x)$,}
\end{equation}
and
\begin{equation}\label{H11}
C(f^n(x)) \le C(x)e^{\epsilon \lvert n\rvert} \quad \text{and} \quad K(f^n(x))\le K(x)e^{\epsilon \lvert n\rvert}.
\end{equation}
We note that the existence of the function $K$ can be easily  deduced from the appropriate bounds for the angles between $\oplus_{j=1}^i E_j(x)$ and $\oplus_{j=i+1}^k E_j(x)$ as well as $\oplus_{j=1}^{i-1} E_j(x)$ and $E_i(x)$ (see Lemma~\ref{ang} and the proof of Theorem~\ref{t3}).
Thus,
\[
\begin{split}
\lVert \cA(f^{-n}(x), n)v_i\rVert &\le C(f^{-n}(x))e^{(\lambda_i+\epsilon)n}\lVert v_i\rVert \\
&\le C(f^{-n}(x))K(f^{-n}(x))e^{(\lambda_i+\epsilon)n}\lVert v\rVert \\
&\le C(x)K(x)e^{(\lambda_i+3\epsilon)n}.
\end{split}
\]
Hence,  using~\eqref{H2} and~\eqref{H8} we obtain
\[
\limsup_{n\to \infty}  \frac 1 n\log \lVert \cA^*(x,n)u\rVert \le \lambda_i+3\epsilon.
\]
Since $\epsilon >0$ was arbitrary, we conclude that
\begin{equation}\label{H12}
\limsup_{n\to \infty}  \frac 1 n\log \lVert \cA^*(x,n)u\rVert \le \lambda_i.
\end{equation}
Obviously, \eqref{H7} and~\eqref{H12} imply~\eqref{H1}. Now we discuss the case $i=1$. If $\lambda_1> -\infty$ then one can repeat the above arguments and establish~\eqref{H1} for $i=1$ also. If $\lambda_1=-\infty$, then one can repeat the second estimates and obtain that
\[
\limsup_{n\to \infty}  \frac 1 n\log \lVert \cA^*(x,n)u\rVert \le L,
\]
for $u\in F_1(x)$, where $L$ is arbitrary real number. By letting $L\to -\infty$, we establish~\eqref{H1} in this situation also.
\end{proof}

\section{H\"{o}lder continunity of Oseledets splitting}
In this section we prove that the Oseledets subspaces of a H\"{o}lder continuous cocycle $A$ are H\"{o}lder continuous on a set of arbitrarily large measure in $X$.
For a  subspace $A\subset \mathbb R^d$ and a vector  $v\in \mathbb R^d$ we define
\[
d(v,A)=\inf \{\lVert v-w\rVert: w\in A\}.
\]
Furthermore, for two subspaces $V$ and $W$ of $\mathbb R^d$ we define the distance between them by
\[
d(V,W)=\max \bigg{\{}\sup_{w\in W, \lVert w\rVert=1} d(w, V), \sup_{v\in V, \lVert v\rVert=1} d(v, W)\bigg{\}}.
\]
It turns out that the quantity $d(V,W)$ can be expressed in an equivalent  form which will be more suitable for our purposes.  Let $P_V \colon \mathbb R^d \to \mathbb R^d$ and $P_W \colon \mathbb R^d \to \mathbb R^d$ be orthogonal projections onto $V$ and $W$ respectively.  The following lemma is well-known;  see e.g.\ page 111 in~\cite{AG}.
\begin{lemma}\label{l1}
For any two subspaces $V,$ and $W$ of $\R^d$ we have that
\[
d(V, W)=\lVert P_V-P_W\rVert.
\]
\end{lemma}
Let $X$ now be a metric space with a metric $\rho$ and let $\Lambda \subset X$.  We say that the family   $E(x)$, $x\in \Lambda$ of subspaces of $\R^d$  is  \emph{H\"{o}lder continuous} on $\Lambda $ if there exist $L, \epsilon_0>0$ and $\beta \in (0, 1]$ such that
\[
d(E(x), E(y))\le L\rho(x, y)^{\beta},  \quad \text{for every $x, y\in \Lambda$ satisfying $\rho(x, y)\le \epsilon_0.$}
\]
We now introduce the notion of H\"{o}lder continuous cocycle. We say that the cocycle $\cA$ is \emph{H\"{o}lder continuous} if there exist $C, \nu >0$ such that
\[
\lVert A(x)-A(y)\rVert \le C\rho(x,y)^{\nu}, \quad \text{for $x, y\in X$.}
\]

The following two simple lemmas will be particulary useful.

\begin{lemma}\label{l2}
Let $\Lambda \subset X$. A family $E(x)$, $x\in \Lambda$ of subspaces of $\R^d$ is H\"{o}lder continuous if and only if the family $E(x)^{\perp}$, $x\in \Lambda$ is  H\"{o}lder continuous.
\end{lemma}

\begin{proof}
Let $P(x)$ denote the orthogonal  projection onto $E(x)$ for $x\in \Lambda$. Then, $\Id-P(x)$ is an orthogonal projection onto $E(x)^{\perp}$. Hence, it follows from  Lemma~\ref{l1} that
\[
\begin{split}
d(E(x), E(y))=\lVert P(x)-P(y)\rVert &= \lVert (\Id-P(x))-(\Id-P(y))\rVert  \\
&=d(E(x)^{\perp}, E(y)^{\perp}),
\end{split}
\]
for every $x, y\in \Lambda$. The conclusion of the lemma now follows directly from the definition of the H\"{o}lder continuity.
\end{proof}

\begin{lemma}\label{l3}
Let $\Lambda \subset X$  and assume that $E(x)$, $F(x)$, $x\in \Lambda$ are two families of subspaces of $\R^d$ such that:
\begin{enumerate}
\item  the subspaces $E(x)$ and $F(x)$ are orthogonal for each $x\in \Lambda$;
\item the family $E(x)$, $x\in \Lambda$ is H\"{o}lder continuous;
\item  the family $E(x)\oplus F(x)$, $x\in \Lambda$ is H\"{o}lder continuous.
\end{enumerate}
Then,  the family $F(x)$, $x\in \Lambda$ is also  H\"{o}lder continuous.
\end{lemma}

\begin{proof}
Let $P(x)$ be an orthogonal projection onto $E(x)$ and let $Q(x)$ be an orthogonal projection onto $F(x)$ for $x\in \Lambda$. Since $E(x)$ and $F(x)$ are orthogonal, we have that $P(x)+Q(x)$ is an orthogonal projection onto $E(x)\oplus F(x)$. Hence, it follows from Lemma~\ref{l1} that
\[
\begin{split}
d(F(x), F(y)) = \lVert Q(x)-Q(y)\rVert &\le \lVert (P(x)+Q(x))-(P(y)+Q(y))\rVert  \\
&\phantom{\le}+\lVert P(x)-P(y)\rVert \\
&=d(E(x)\oplus F(x), E(y)\oplus F(y)) \\
&\phantom{=}+d(E(x), E(y)),
\end{split}
\]
for every $x, y\in \Lambda$.
The H\"{o}lder continuity of the family $F(x)$ follows directly from the H\"{o}lder continuity of families $E(x)\oplus F(x)$ and $E(x)$.
\end{proof}

We will  use the following two auxiliary results from~\cite{ABF} which are slight generalisations of the original work of Brin~\cite{B} (see also Lemmas 5.3.4. and 5.3.5. in~\cite{BP}).

\begin{lemma}\label{Brin}
Let $(A_n)_{n\ge 1}$, $(B_n)_{n\ge 1}$ be two sequences of real matrices of order $d>0$ such that for some $0< \lambda <\mu$ and $C\ge 1$ there exist subspaces $E, E', F, F'$ of $\R^d$ satisfying
$\R^d=E\oplus E'=F\oplus F'$ such that:
\begin{enumerate}
\item $\lVert A_n u\rVert \le C\lambda^n \lVert u\rVert$  for $u\in E$ and $C^{-1} \mu^n \lVert v\rVert \le \lVert A_n v\rVert$ for $v\in E'$;
\item $\lVert B_n u\rVert \le C\lambda^n \lVert u\rVert$  for $u\in F$ and $C^{-1} \mu^n \lVert v\rVert \le \lVert B_n v\rVert$ for $v\in F'$;
\item $\max \{\lVert v\rVert, \lVert w\rVert \} \le d\lVert v+w\rVert$ for $v\in E$, $w\in E'$ or $v\in F$, $w\in F'$.
\end{enumerate}
Then for each pair $(\delta, a) \in (0, 1] \times [\lambda, +\infty)$ satisfying
\[
\bigg{(}\frac{\lambda}{a}\bigg{)}^{n+1} < \delta \le \bigg{(}\frac{\lambda}{a}\bigg{)}^n \quad \text{and} \quad \lVert A_n-B_n \rVert \le \delta a^n,
\]
we have that $d(E, F) \le (2+d)C^2\frac{\mu}{\lambda} \delta^{\log (\mu/ \lambda) / \log (a/ \lambda)}$.
\end{lemma}

\begin{lemma}\label{Brin2}
Assume that $\cA$ is a H\"{o}lder continuous cocycle and that there exists $L>0$ such that $f$ is Lipschitz with constant $L$ and such that $\lVert \cA(x, n)\rVert \le L^n$ for $n\ge 0$ and $x$ in some fixed compact set $\Lambda \subset X$.  Then, there exist $C, \nu >0$ such that
$\lVert \cA(x,n)-\cA(y,n)\rVert \le C^n d(x,y)^{\nu}$ for $x, y\in \Lambda$ and $n\ge 0$.
\end{lemma}

We now arrive at our main result.
\begin{theorem}\label{main}
Let $\cA$ be a H\"{o}lder continuous cocycle satisfying~\eqref{G} whose Lyapunov exponents and the corresponding  Oseledets splitting are given by~\eqref{LE} and~\eqref{OS}. Then, for each $i\in \{1, \ldots, k\}$ and  $\delta >0$ there exists a compact set $\Lambda \subset X$ of measure $\mu(\Lambda)>1-\delta$ such that
the map $x\mapsto E_i(x)$ is H\"{o}lder continuous on $\Lambda$.
\end{theorem}

\begin{proof}
We will divide the proof into several parts.

\textsl{Step 1 -- H\"older continuity of $x \mapsto E_1(x)\oplus \ldots \oplus E_i(x)$:}
We will first show that for each $i\in \{1, \ldots, k\}$ and $\delta >0$, there exists a compact set $\Lambda \subset X$ satisfying  $\mu(\Lambda) > 1-\delta$ and such that  the map  $x\mapsto E_1(x)\oplus \ldots \oplus E_i(x)$ is H\"{o}lder continuous on $\Lambda$. We note that this part is essentially already contained in~\cite{B}, \cite{BP} and~\cite{ABF}; we include it for the sake of completeness.

 Choosing $\epsilon >0$ such that $\lambda_i+\epsilon < \lambda_{i+1}-\epsilon$, it follows from Theorem~\ref{t3} that there exists a Borel set $\Lambda \subset X$ satisfying  $\mu(\Lambda)=1$ and   Borel measurable functions $C, K \colon \Lambda \to (0, \infty)$ such that~\eqref{1},
\eqref{2}, \eqref{A} and~\eqref{3} hold (with the same convention as in the statement of Theorem~\ref{t3} for $\lambda_1$ in~\eqref{1} if $i=1$). Moreover, there exists a  measurable function $\tilde C \colon  \Lambda \to (0, \infty)$ such that
\begin{equation}\label{Z}
\lVert \cA(x, n)\rVert \le \tilde C(x)e^{(\lambda_k+\epsilon)n}, \quad \text{for $x\in  \Lambda$ and $n\ge 0$.}
\end{equation}
For $l\in \N$, let
\[
\Lambda_l=\{x\in \Lambda  : C(x) \le l, \  \tilde C(x) \le l, \ K(x) \le l \}.
\]
It is easy to verify (see~ p.121 \cite{BP}) that each set $\Lambda_l$ is compact. Furthermore, we obviously have that
\[
\Lambda_l \subset \Lambda_{l+1} \quad \text{and} \quad \bigcup_{l=1}^\infty \Lambda_l=\Lambda.
\]
Since $\mu (\Lambda)=1$,  there exists $l\in \N$ such that  $\mu(\Lambda_l)>1-\delta$.  On the other hand, it follows from~\eqref{1}, \eqref{2}, \eqref{A} and~\eqref{3} that for every $x\in \Lambda_l$,
\begin{equation}\label{4}
\lVert \cA(x,n)v\rVert \le le^{(\lambda_i+\epsilon)n} \lVert v\rVert, \quad \text{for $v\in E_1(x) \oplus \ldots \oplus E_i(x)$ and $n\ge 0$;}
\end{equation}
\begin{equation} \label{5}
\frac{1}{l} e^{(\lambda_{i+1}-\epsilon)n} \lVert v\rVert \le \lVert \cA(x,n)v\rVert,  \quad \text{for $v\in E_{i+1}(x)\oplus \ldots \oplus E_k(x)$ and $n\ge 0$;}
\end{equation}
and
\begin{equation}\label{6}
\lVert u\rVert \le l \lVert u+v\rVert \quad  \text{and} \quad \lVert v\rVert \le l \lVert u+v\rVert,
\end{equation}
for $u\in E_1(x)\oplus \ldots \oplus E_i(x)$ and $v\in E_{i+1}(x)\oplus \ldots \oplus E_k(x)$. Moreover, it follows from~\eqref{Z}  that
\begin{equation}\label{R}
\lVert \cA(x,n)\rVert \le le^{(\lambda_k+\epsilon)n}, \quad \text{for every $x\in \Lambda_l$ and $n\ge 0$.}
\end{equation}
Hence, applying Lemmas~\ref{Brin} and~\ref{Brin2} to sequences $A_n=\cA(x, n)$ and $B_n=\cA(y, n)$ with $x, y\in \Lambda_l$, we easily obtain the H\"{o}lder continuity of the map $x\mapsto E_1(x)\oplus \ldots \oplus E_i(x)$ on $\Lambda_l$.

\textsl{Step 2 -- H\"older continuity of $x\mapsto E_{i+1}(x)\oplus \ldots \oplus E_k(x)$:}
We now prove that for each $i\in \{1, \ldots, k-1\}$ and $\delta >0$, there exists a compact  set $\Lambda \subset X$ satisfying  $\mu(\Lambda) > 1-\delta$ and such that  the map  $x\mapsto E_{i+1}(x)\oplus \cdots \oplus E_k(x)$ is H\"{o}lder continuous on $\Lambda$.
We note that
$$\bigg{(} E_{i+1} \oplus \cdots \oplus E_k(x) \bigg{)}^{\perp}=\bigg{(} \bigoplus_{j\neq 1} E_j(x) \bigg{)}^{\perp} \oplus \cdots \oplus  \bigg{(} \bigoplus_{ j\neq i} E_j(x) \bigg{)}^{\perp}=F_1\oplus\cdots\oplus F_i,$$
where $F_j, j=1,\ldots, i$ are the Oseledets spaces of the adjoint cocycle and the second equality follows from Theorem \ref{t2}.
We now apply Step 1 to the adjoint cocycle to find a measurable set $\Lambda \subset X$ satisfying  $\mu(\Lambda) > 1-\delta$ and such that the map
$$
x \mapsto  F_1\oplus\cdots\oplus F_i= \bigg{(} E_{i+1} \oplus \cdots \oplus E_k(x) \bigg{)}^{\perp}$$
is H\"{o}lder continuous on $\Lambda$.
The desired conclusion now follows directly from Lemma~\ref{l2}.
We note that this part is also contained in~\cite{B}, \cite{BP} and~\cite{ABF}, but is obtained by applying the previous step to the inverse of $\cA$.
Such an approach is unavailable to us because the inverse of $\cA$ may not exist.

\textsl{Step 3 -- H\"older continuity of $x\mapsto E_i(x)$:}
In the final part, we prove that for each $i\in \{1, \ldots, k\}$ and $\delta >0$, there exists a compact set $\Lambda \subset X$ satisfying  $\mu(\Lambda) > 1-\delta$ and such that  the map $x\mapsto E_i(x)$ is H\"{o}lder continuous on $\Lambda$. We note that our arguments related to this part of the proof differ significantly from those in~\cite{ABF} and are arguably
simpler. For $i=1$ or $i=k$ there is nothing to prove since the conclusion follows directly from the previous two steps. Take now an arbitrary $i\in \{2, \ldots, k-1\}$ and  $\delta >0$ and let $\Lambda_l$ be the set of measure greater then $1-\delta$ constructed in the first step of the proof. Without loss of the generality we can assume that the mapping $x\mapsto E_i(x) \oplus \ldots \oplus E_k(x)$
is H\"{o}lder continuous on $\Lambda_l$ since otherwise (using Step 2 of the proof) we can pass to a compact subset of $\Lambda_l$ of measure greater then $1-2\delta$ on which this holds.
Let $P(x)$ denote the orthogonal projection onto $F(x):=E_i (x) \oplus \ldots \oplus E_k(x)$.
 For a  point $x\in \Lambda_l$, we define a sequence of matrices $(A_n)_n$ by $A_n=\cA(x,n)P(x)$.
Choose any  $v\in F(x)^{\perp} \oplus E_i(x)$ and write it in the form  $v=v_1+v_2$ where $v_1 \in F(x)^{\perp}$ and $v_2\in E_i(x)$. Using~\eqref{4} and the orthogonality of $F(x)^{\perp}$ and $E_i(x)$,
we conclude that
\[
 \lVert A_n v\rVert =\lVert A_n v_2\rVert =\lVert \cA(x,n)v_2\rVert \le le^{(\lambda_i+\epsilon)n} \lVert v_2\rVert \le le^{(\lambda_i+\epsilon)n} \lVert v\rVert.
\]
Hence,
\begin{equation}\label{N1}
\lVert A_n v\rVert \le l e^{(\lambda_i+\epsilon)n} \lVert v\rVert, \quad \text{for every $v\in F(x)^{\perp} \oplus E_i(x) $ and $n\ge 0$.}
\end{equation}
On the other hand, it follows readily from~\eqref{5} that
\begin{equation}\label{N2}
 \frac{1}{l} e^{(\lambda_{i+1}-\epsilon)n} \lVert v\rVert \le \lVert A_n v\rVert, \quad \text{for every  $v\in E_{i+1}(x)\oplus \ldots \oplus E_k(x)$ and $n\ge 0$.}
\end{equation}
 We now want to establish appropriate bounds for the angles between  $F(x)^{\perp} \oplus E_i(x)$ and $E_{i+1}(x)\oplus \ldots \oplus E_k(x)$.
 Select an arbitrary
$u=u_1+u_2\in F(x)^{\perp} \oplus E_i(x)$ with $u_1\in F(x)^{\perp}$, $u_2\in E_i(x)$, and $v\in E_{i+1}(x)\oplus \ldots \oplus E_k(x)$. By~\eqref{6}, we have  that
\[
 \lVert u+v\rVert^2=\lVert u_1+u_2+v\rVert^2=\lVert u_1\rVert^2+\lVert u_2+v\rVert^2\ge \lVert u_2+v\rVert^2 \ge \frac{1}{l^2}\lVert v\rVert^2,
\]
which implies that
\[
l \lVert u+v\rVert \ge  \lVert v\rVert.
\]
Similarly,
\[
 \lVert u+v\rVert^2=\lVert u_1\rVert^2+\lVert u_2+v\rVert^2 \ge \frac{1}{l^2}(\lVert u_1\rVert^2+\lVert u_2\rVert^2)=\frac{1}{l^2}\lVert u\rVert^2,
\]
which implies that
\[
l \lVert u+v\rVert \ge \lVert u\rVert.
\]
Hence,
\begin{equation}\label{N3}
\max \{\lVert u\rVert, \lVert v\rVert \} \le l\lVert u+v\rVert, \quad \text{for $u\in F(x)^{\perp} \oplus E_i(x)$ and $v\in \bigoplus_{j=i+1}^kE_j(x)$.}
\end{equation}
In order to apply Lemma~\ref{Brin}, we need to  bound the quantity $\lVert A_n-B_n\rVert$, where $B_n=\cA(y,n)P(y)$ and $y$ is some other point in $\Lambda_l$. We have
\begin{eqnarray*}
\lefteqn{\lVert A_n-B_n\rVert}\\
&=&\lVert \cA(x,n)P(x)-\cA(y,n)P(y)\rVert \\
&\le&\lVert \cA(x,n)P(x)-\cA(x,n)P(y)\rVert  +\lVert\cA(x,n)P(y)-\cA(y,n)P(y)\rVert \\
&\le&\lVert \cA(x,n)\rVert \cdot  d(F(x), F(y))+\lVert \cA(x,n)-\cA(y,n)\rVert.
\end{eqnarray*}
By applying Lemma~\ref{Brin2}, \eqref{R} and  the H\"{o}lder continuity of the map $x\mapsto F(x)$  we conclude that there exists $C, \nu >0$ (independent of $x$ and $y$)
 such that
\[
 \lVert A_n-B_n\rVert \le C^n d(x,y)^\nu.
\]
It now follows from Lemma~\ref{Brin} that the map $x\mapsto F(x)^{\perp} \oplus E_i(x)$ is H\"{o}lder continuous  on a set $\Lambda_l$, which in view of Lemma~\ref{l3}  implies the H\"{o}lder continuity of the map
$x\mapsto E_i(x)$ on the same set.

\end{proof}

\section{Cocycles on Hilbert spaces}
In this section we state a generalization of our results to cocycles on Hilbert spaces.
\begin{theorem}
Let $X$ be a Borel subset of a separable complete metric space, $f:X\circlearrowleft$ be a bi-Lipschitz ergodic transformation, $\mathcal{H}$ a Hilbert space, and $A:X\to \mathcal B(\mathcal{H})$ take values in the space of all compact operators.
Further, assume that $A$ satisfies (\ref{G}) and that $x\mapsto A(x)$ is H\"older continuous in the operator norm topology.
Then either:
\begin{enumerate}
\item
There is a finite sequence of numbers (note that below we order Lyapunov exponents and Oseledets spaces in reverse order to (\ref{LE}) and (\ref{OS}) starting with the largest one)
\[
\lambda_1 >\lambda_2 >\cdots >\lambda_k>\lambda_\infty=-\infty
\]
and a decomposition
\[
\mathcal{H}=E_1(x) \oplus \cdots \oplus E_k(x) \oplus E_{\infty}(x)
\]
such that
\[
A(x) E_i(x) = E_i(f(x)), i=1,\ldots,k\quad\mbox{and}\quad A(x) E_\infty(x) \subset E_\infty(f(x))  
\]
and
\[
\lim_{n\to \infty} \frac 1 n \log \lVert \cA(x,n)v\rVert=\lambda_i, \quad \text{for $v\in E_i(x)\setminus \{0\}$,  $i\in \{1,\ldots,k\}\cup \{\infty\}$.}
\]
Moreover, each $E_i(x), i=1,\ldots,k$ is a finite-dimensional subspace of $\mathcal{H}$.
The maps $x\mapsto E_i(x)$, $i=1,\ldots,k$ are H\"{o}lder continuous on a compact set of arbitrarily large measure.
\item There exists an infinite sequence of numbers
\[
\lambda_1 >\lambda_2 >\cdots >\lambda_k > \ldots  >\lambda_{\infty}=-\infty
\]
and a decomposition
\[
\mathcal{H}=E_1(x) \oplus \cdots \oplus E_k(x) \oplus \cdots \oplus E_{\infty}(x)
\]
such that
\[
A(x) E_i(x) = E_i(f(x)), 1\le i<\infty\quad\mbox{and}\quad A(x) E_\infty(x) \subset E_\infty(f(x))
\]
and
\[
\lim_{n\to \infty} \frac 1 n \log \lVert \cA(x,n)v\rVert=\lambda_i, \quad \text{for $v\in E_i(x)\setminus \{0\}$,  $i\in \N \cup \{\infty\}$.}
\]
Moreover, each $E_i(x), i\neq \infty$ is a finite-dimensional subspace of $\mathcal{H}$.
The maps $x\mapsto E_i(x)$, $i\neq \infty$ are H\"{o}lder continuous on a compact set of arbitrarily large measure.
\end{enumerate}
\end{theorem}
\begin{proof}
In cases 1. and 2., all statements except the H\"older continuity follow from Theorem 17 in \cite{FLQ2}, which extends Theorem~\ref{t1} to a semi-invertible version of Oseledets' theorem under weaker conditions than those in our hypotheses, namely $f$ is an ergodic homeomorphism, $\mathcal{H}$ is a Banach space, $A$ takes values in the space of all quasi-compact operators, and $x\mapsto A(x)$ is $\mu$-continuous;  the resulting Oseledets splitting is $\mu$-continuous.

\textsl{Case 1:} In this case one is able to repeat the arguments in the proof of Theorem~\ref{main} and establish the H\"older continuity.
Indeed, we emphasize that all of the preparatory results that we used extend to this setting.
In particular, Theorem~\ref{t2} and Lemma \ref{Brin} are valid for cocycles on Hilbert space (with the same proof) and the corresponding version of  Theorem~\ref{t3} also holds. More precisely, one is able to repeat all  arguments in the proof Theorem~\ref{t3} with the exception of Lemma~\ref{9216}.

The statement of Lemma~\ref{9216} remains unchanged in our Hilbert space setting.
The proof in the Hilbert space case is identical to the proof of Proposition 14 in \cite{FLQ2}, where Proposition 14 is now applied to the
cocycle $\mathcal{A}(x,n)$ restricted to $E^1(x)$.

\textsl{Case 2:} Arguing as in case 1.\ above, one is able to show H\"{o}lder continuity of Oseledets subspaces $E_i(x)$ for $i\neq \infty$ on a compact set of arbitrarily large measure for any
H\"{o}lder continuous cocycle.
Unfortunately, we are not able to establish a similar property for $E_{\infty}(x)$.
The major obstacle in repeating our arguments is the fact that one doesn't have a lower bound for the expansion on a subspace complementary to $E_{\infty}(x)$.
\end{proof}

\section{Acknowledgements}
We thank Ian Morris for pointing us to the reference \cite{AG} and Anthony Quas for useful discussions.  We also thank an anonymous referee for their careful reading of the manuscript.

\end{document}